\documentclass[a4paper,12pt]{article}

\usepackage{amsmath}
\usepackage{amsfonts}
\usepackage{amssymb}
\usepackage{amstext    }
\usepackage{amsthm    }

\usepackage{mathrsfs}

\usepackage{hyperref}

\usepackage{color}

\usepackage[T1]{fontenc}
\usepackage{graphicx}

\usepackage{epsf}
\usepackage{psfrag}

\usepackage{fancyhdr}

\usepackage{url}

\newtheorem{theorem}{Theorem}[section]
\newtheorem{lemma}[theorem]{Lemma}

\newtheorem{remark}[theorem]{Remark}

\newcommand{\cali}[1]{\mathscr{#1}}

\numberwithin{equation}{section}

\newcommand{\ddc}{{dd^c}}
\newcommand{\dbar}{{\overline\partial}}
\newcommand{\ddbar}{{\partial\overline\partial}}

\newcommand{\supp}{{\rm supp}}
\newcommand{\id}{{\rm Id}}

\newcommand{\Cc}{\cali{C}}

\newcommand{\Rc}{\cali{R}}
\newcommand{\Sc}{\cali{S}}

\title{Speed of convergence towards attracting sets for endomorphisms of
$\mathbb P^k$}
\author{Johan Taflin}

\begin{document}

\maketitle

\begin{abstract}
Let $f$ be a non-invertible holomorphic endomorphism of $\mathbb P^k$ having an
attracting set $A.$ We show that, under some natural assumptions, $A$ supports a
unique invariant positive closed current $\tau,$ of the right bidegree and of
mass $1.$ Moreover, if $R$ is a current supported in a small neighborhood of $A$
then its push-forwards by $f^n$ converge to $\tau$ exponentially fast. We also
prove that the equilibrium measure on $A$ is hyperbolic.
\end{abstract}
\section{Introduction}\label{sec intro}
Let $f$ be a holomorphic endomorphism of algebraic degree $d\geq2$ on the
complex projective space $\mathbb P^k.$ A compact subset $A$ of $\mathbb P^k$
is called an \textit{attracting set} if it has a \textit{trapping neighborhood}
$U$ i.e. $f(U)\Subset U$ and $A=\cap_{n\geq0}f^n(U)$ where
$f^n:=f\circ\cdots\circ f,$ $n$ times. It follows that $A$ is invariant,
$f(A)=A.$
Furthermore, if $A$ contains a dense orbit then $A$ is a \textit{trapped
attractor}.
Typical examples of such objects are fractal and their underlying dynamics are
hard to study. We refer to \cite{mil-attractor}, \cite{ruelle-book} for general
discussions on attractors and to \cite{fw-attractor},
\cite{jonsson-weickert-attractor}, \cite{fs-example},
\cite{bdm-elliptic}, \cite{t-elliptic} and references therein for examples of
different types of attractors in $\mathbb P^2.$

Attracting sets are stable under small perturbations. Indeed, if $f$ has an
attracting set $A=\cap_{n\geq0}f^n(U)$ then any small perturbation $f_\epsilon$
of $f$ has an attracting set defined by
$A_\epsilon=\cap_{n\geq0}f_\epsilon^n(U).$ For example, when $f$ restricted to
$\mathbb
C^k$ defines a polynomial self-map then the hyperplane at infinity $\mathbb
P^k\setminus\mathbb C^k$ is an
attracting set. In the same way, it is easy to create examples where the
attracting set is a projective subspace of arbitrary dimension. In this paper,
we
consider a family of endomorphisms, stable under small perturbations, which
contains these examples. It was introduced by Dinh in \cite{d-attractor} and we
briefly recall the context.

In the sequel, we always assume that $f$ possesses an attracting set $A$ which
has a trapped neighborhood $U$ satisfying the following properties. There exist
an integer $1\leq p \leq k-1$ and two projective subspaces $I$ and $L$ of
dimension
$p-1$ and $k-p$ respectively such that $I\cap U=\varnothing$ and $L\subset U.$
We do not assume that $L$ and $I$ are invariant.
Since $I\cap L=\varnothing,$ for each $x\in L$ there exists a unique
projective subspace $I(x)$ of dimension $p$ which contains $I$ and such that
$L\cap
I(x)=\{x\}.$ Furthermore, for each $x\in L$ we ask that $U\cap I(x)$ is strictly
convex as a subset of $I(x)\setminus I\simeq \mathbb C^p.$ All these
assumptions are stable
under small perturbations of $f.$ The geometric
assumption on $U$ is slightly stronger than the one of Dinh, who only requires
$U\cap I(x)$ to be star-shaped at $x.$ We need convexity in order the solve the
$\overline\partial$-equation on $U.$ Indeed, under our assumption $U$ is a
$(p-1)$-convex domain in the sense of \cite{henkin-leiterer}.

If $E$ is a subset of $\mathbb P^k,$ let $\Cc_q(E)$ denote the set of all
positive closed currents of bidegree
$(q,q),$ supported in $E$ and of mass $1.$ 
It is well known that for any integer $1\leq q\leq k$ and any smooth form $S$
in $\Cc_q(\mathbb P^k),$ the sequence
$d^{-qn}(f^n)^*(S)$ converges to a positive closed current $T^q$ of
bidegree $(q,q)$ called the \textit{Green current of order $q$} of $f.$ We
refer to \cite{ds-lec} for a detailed exposition on these currents and their
effectiveness in holomorphic dynamics.

When $q=k,$ it gives the
\textit{equilibrium measure} of $f,$ $\mu:=T^{k}.$ It is
exponentially mixing and it is the unique measure of maximal entropy $k\log d$
on $\mathbb P^k.$ Moreover, it is hyperbolic and all its Lyapunov exponents
are larger or equal to $(\log d)/2.$ The dynamics outside the support of $\mu$
is not very well
understood. The aim of this paper is to continue the investigation started in
\cite{d-attractor} on the attracting sets described above, which do not
intersect $\supp(\mu).$ Indeed, since $I\cap U=\varnothing,$ by
regularization there exists a smooth form
$S\in\Cc_{k-p+1}(\Omega),$ where $\Omega:=\mathbb P^k\setminus \overline U.$
As $f^{-1}(\Omega)\subset \Omega,$ it follows that $\supp(T^{k-p+1})\cap
U=\varnothing,$ and hence $\supp(T^q)\cap U=\varnothing$ if $q\geq k-p+1.$

The set $\Cc_p(U)$ is non-empty since it contains
the current $[L]$ of integration on $L$ and its regularizations in $U.$ In the
situation described above, Dinh proved that if $R$ is a continuous element of
$\Cc_p(U)$ then its normalized push-forwards by $f^n,$ $d^{-(k-p)n}(f^n)_*(R),$
converge to a current $\tau$ which is independent of the choice of $R.$
Moreover, the current $\tau$ gives us information on the geometry of $A$ and
on the dynamics of $f_{|A}$: it is
woven, supported in $A$ and invariant i.e. $f_*(\tau)=d^{k-p}\tau.$ Our
main result is that, with a natural additional assumption on $f_{|U},$ stable
under
small perturbations, we obtain an explicit exponential speed of the above
convergence
for any
$R$ in $\Cc_p(U).$

\begin{theorem}\label{th main}
Let $f$ and $\tau$ be as above and assume that
$\|\wedge^{k-p+1}Df(z)\|<1$
on $\overline U.$ There is a
constant $0<\lambda<1$ such that for each $0<\alpha\leq2$ the following
property holds. There exists $C>0$
such that for any element $R$ of $\Cc_p(U)$ and
any $(k-p,k-p)$-form $\varphi$ of class $\mathcal C^\alpha$ on $\mathbb P^k$ we
have 
\begin{equation}\label{eq main}
|\langle d^{-(k-p)n}(f^n)_*(R)-\tau,\varphi\rangle|\leq
C\lambda^{n\alpha/2}\|\varphi\|_{
\mathcal C^\alpha}.
\end{equation}
In particular, $\tau$ is the unique invariant current in $\Cc_p(U)$ and
$d^{-(k-p)n}(f^n)_*(R)$ converge to $\tau$ uniformly on $R\in\Cc_p(U).$
\end{theorem}
Recall that $f$ induces a self-map $Df$ on the tangent bundle $T\mathbb P^k$
which also gives a self-map $\wedge^qDf$ on the exterior power
$\wedge^qT\mathbb P^k,$ $1\leq q\leq k.$
In the sequel, all the norms on $\mathcal C^\alpha,$ $L^r,$ etc. are with
respect to the Fubini-Study metric on $\mathbb P^k.$ It also gives a uniform
norm which induces an operator norm for $\wedge^qDf.$

In the
same spirit as Theorem \ref{th main}, we proved in \cite{t-vitesse} that for a
generic current $S$ in $\Cc_1(\mathbb P^k),$ the sequence $d^{-n}(f^n)^*(S)$
converges to $T$ exponentially fast. However, the contexts are quite different.
Here, we consider currents of arbitrary bidegree which are in general much
harder to handle. Moreover, in \cite{t-vitesse} we deeply use that $T$ has
Hölder continuous local potentials. In the present situation, we can expect that
the attracting current $\tau$ is always more singular. The idea to
prove Theorem \ref{th main} is to use Henkin-Leiterer's solution with estimates
of the $\ddc$-equation on $U$ in order to study separately the harmonic and
non-harmonic
parts of the left hand side term of \eqref{eq main}. When $\ddc \varphi=0$ on
$U,$ we use the ``geometry''
of $\Cc_p(U),$ introduced in \cite{d-attractor} and \cite{ds-geo}, and Harnack's
inequality to obtain exponential estimates. In order to deal with the
non-harmonic
part, we use the assumption on $\|\wedge^{k-p+1}Df\|.$ This
assumption comes naturally in several basic examples and their
perturbations.

In \cite{d-attractor}, Dinh also showed that the \textit{equilibrium measure
associated to $A$}, defined by
$\nu:=\tau\wedge T^{k-p},$ is invariant, mixing and of maximal entropy
$(k-p)\log d$ on
$A.$ Theorem \ref{th main} is a first step in order to obtain other ergodic and
stochastic properties on $\nu$ as exponential mixing or central limit
theorem. We postpone this question in a future work.

Under the same assumptions, we deduce from the work of de Thélin
\cite{dethelin-lyap}, see also \cite{dupont-large-entropy}, the following result
on $\nu.$

\begin{theorem}\label{th hyp}
If $f$ is as in Theorem \ref{th main}, then the measure $\nu$ is hyperbolic.
More precisely, counting with multiplicity it has $k-p$ Lyapunov exponents
larger than or equal to $(\log d)/2$ and $p$
Lyapunov exponents strictly smaller than $-(k-p)(\log d)/2.$
\end{theorem}

\section{Structural discs of currents}\label{sec disc}
In this section we recall the notion of structural varieties of currents. It was
introduced by Dinh and Sibony in order to put a geometry on the
space $\Cc_p(U)$ which is of infinite dimension, see \cite{ds-geo}
and \cite{d-attractor}. The definition of structural varieties is based on
slicing theory and they can be seen as
complex subvarieties inside $\Cc_p(U).$ In \cite{ds-superpot}, the authors
developed the notion of super-potential which involves more deeply this
geometry.

Slicing theory can be seen as a generalization to currents of restriction of
smooth forms to submanifolds. We will briefly explain it in our context and
refer to \cite{fed-book} for a more complete account.

Let $U$ be an open subset of $\mathbb P^k$ satisfying the geometric hypothesis
as above. Let $V$ be a complex manifold of
dimension $l.$ We denote by $\pi_U$ and $\pi_V$ the canonical projections of
$U\times V$ to $U$ and $V$ respectively. If $\Rc$ is a positive closed current
of bidegree $(p,p)$ on $U\times V$ with $\pi_U(\supp(\Rc))\Subset
U$ then, for all $\theta$ in $V$, the slice
$\langle\Rc,\pi_V,\theta\rangle$ is well defined. For any
$(k-p,k-p)$-form $\phi$ on $U\times V$ we have
$$\langle \Rc,\pi_V,\theta\rangle(\phi)=\lim_{\epsilon\to 0}\langle
\Rc\wedge\pi_V^*(\psi_{\theta,\epsilon}),\phi\rangle,$$
where $\psi_{\theta,\epsilon}$ is an appropriate approximation in $V$ of the
Dirac mass at $\theta.$ It is a $(p+l,p+l)$-current
on $U\times V$ supported on $U\times\{\theta\}$ which can be identified
to a $(p,p)$-current on $U.$ A family of currents $(R_\theta)_{\theta\in V}$ in
$\Cc_p(U)$ is a \textit{structural variety} if there exists a positive closed
current $\Rc$ in $U\times V$ such that $R_\theta=\langle
\Rc,\pi_V,\theta\rangle.$ When $V$ is isomorphic to the unit disc of $\mathbb
C,$ we call $(R_\theta)_{\theta\in V}$ a \textit{structural disc}.

Recall that in our situation $f(U)\Subset U.$ Under the geometrical assumption
on
$U,$ Dinh constructed in \cite[p.233]{d-attractor} a family of structural discs
in $\Cc_p(U).$ He uses that for each $x\in L$ the set $I(x)\cap U$ is
star-shaped at $x$ to obtain a property similar to star-sharpness for
$\Cc_p(U).$ 

More precisely, up to an automorphism, we can assume that
$$I=\{x\in\mathbb P^k\,|\,x_i=0,\ 0\leq i\leq k-p\},\ L=\{x\in\mathbb P^k\,|\,
x_i=0,\ k-p+1\leq i\leq k\},$$
where $x=[x_0:\cdots:x_k]$ are the homogeneous coordinates of $\mathbb P^k.$
For $\theta\in\mathbb C,$ $A_\theta(x):=[x_0:\cdots:x_{k-p}:\theta
x_{k-p+1}:\cdots:\theta x_k]$ is an automorphism of $\mathbb P^k$ except for
$\theta=0$ where it is the projection of $\mathbb P^k\setminus I$ on $L.$ Let
set $U':=f(U).$ As $I(x)\cap U$ is star-shaped at $x\in L,$ there exists a
simply connected open neighborhood $V\subset \mathbb C$ of $[0,1]$ such that
$A_\theta(U')\Subset U$ for all $\theta$ in $\overline V.$ If $S$ is in
$\Cc_p(U')$ then the family $(S_\theta)_{\theta\in V}$ with
$S_\theta:=(A_\theta)_*(S)$ defined a structural disc. Indeed, if 
$\Lambda:\mathbb P^k\times V\to \mathbb P^k$ is the meromorphic map defined by
$\Lambda(x,\theta)=(A_\theta)^{-1}(x)$ and $\Sc:=\Lambda^*S$ then
$S_\theta=\langle \Sc,\pi_V,\theta\rangle,$ see \cite{d-attractor} for more
details. For any $S$ in $\Cc_p(U'),$ we have that $S_1=S$
and $S_0=[L]$ which is independent of $S.$ In other words, any current in
$\Cc_p(U')$ is linked to $[L]$ by a structural disc in $\Cc_p(U).$ Moreover,
$S_\theta$
depends continuously on $\theta$ and we have the
following important property.

\begin{lemma}[\cite{d-attractor}]\label{le harmo}
Let $S$ be in $\Cc_p(U')$ and $(S_\theta)_{\theta\in V}$ be the structural disc
described
above.
If $\phi$ is a real continuous $(k-p,k-p)$-form with $\ddc\phi=0$ on $U$ then
the
function $\theta\mapsto \langle
S_\theta,\phi\rangle$ is harmonic on $V.$
\end{lemma}

\section{q-Convex set and d-bar equation}
The concept of $q$-convexity generalizes both Stein and compact
manifolds. Andreotti and Grauert \cite{andreotti-grauert} obtained vanishing or
finiteness theorems for $q$-convex manifolds and, in \cite{henkin-leiterer},
Henkin and Leiterer developed a similar theory using integral representations.
In particular, they obtained solutions of the $\dbar$-equation with explicit
estimates, which play a key role in our proof. For this reason, we use the
conventions of \cite{henkin-leiterer}.

If $1\leq q\leq k$ is an integer then a real $\mathcal C^2$ function $\rho$ on
an open subset $V\subset \mathbb P^k$ is called \textit{$q$-convex} if, in
any holomorphic local coordinates, the Hermitian form
$$L_\rho(x)t=\sum_{i,j=1}^k\frac{\partial^2\rho(x)}{\partial z_i\partial
\overline z_j}t_i\overline t_j$$
has at least $q$ strictly positive eigenvalues at any point $x\in V.$

Let $0\leq q\leq k-1.$ We say that an open subset $D$ of $\mathbb
P^k$ is \textit{strictly $q$-convex} if there exists a $(q+1)$-convex function
$\rho$ in a neighborhood $V$ of $\partial D$ such that
$$D\cap V=\{x\in V\,|\, \rho(x)<0\}.$$
Moreover, if the same condition is satisfied with $V$ a neighborhood of
$\overline D$ then $D$ is called \textit{completely strictly $q$-convex}.

The strict $q$-convexity has the following important consequence, see
\cite[Theorem 11.2]{henkin-leiterer}.
\begin{theorem}\label{th hl}
 Let $D$ be a strictly $q$-convex open subset of $\mathbb P^k$ with $\mathcal
C^2$ boundary. If $\phi$ is a continuous $\dbar$-exact form of bidegree $(r,s)$
in a neighborhood of
$\overline D$ with $0\leq s\leq k,$ $k-q\leq r\leq k,$ then there exists a
continuous $(s,r-1)$-form $\psi$ on $D$
such that $\dbar\psi=\phi$ and
$$\|\psi\|_{\infty,D}\leq C\|\phi\|_{\infty,D}$$
for some $C>0$ independent of $\phi.$
\end{theorem}
Furthermore, Andreotti and Grauert proved the following vanishing theorem, see
\cite{andreotti-grauert} and \cite[Theorem 12.7]{henkin-leiterer}.
\begin{theorem}\label{th vanish}
If $D$ is a completely strictly $q$-convex open subset of $\mathbb P^k$ with
$\mathcal C^2$ boundary then $H^{s,r}(D,\mathbb C)=0$ for any $0\leq s\leq k$
and $k-q\leq r\leq k.$
\end{theorem}

Henkin and Leiterer \cite[Theorem 5.13]{henkin-leiterer} give the following
criteria of $q$-convexity, which is closely related to our geometric assumption
on $U$ with $q=p-1.$
\begin{theorem}\label{th qconvex}
Let $D$ be an open subset of $\mathbb P^k$ with $\mathcal C^2$ boundary. If for
each $x\in\partial D$ there exists a complex submanifold $Y\subset \mathbb P^k$
of dimension $q+1,$ transverse to $\partial D$ and such that $Y\cap D$ is a
strictly pseudoconvex domain in $Y,$ then $D$ is
strictly $q$-convex.
\end{theorem}
This result applies to our trapping neighborhood $U$ with $q=p-1.$ Indeed,
observe that, possibly by exchanging $U$ by a slightly smaller open set which
contains $f(U),$ we can assume that $\partial U$ is smooth and the intersection
of $\partial U$ with $I(x)$ is transverse for all $x\in L.$ The projective
space $I(x)$ has dimension $p=q+1$ and $U\cap I(x)$ is strictly convex in
$I(x)\setminus I\simeq \mathbb C^p,$ so in particular strictly pseudoconvex in
$I(x).$ Therefore, by Theorem \ref{th qconvex}, $U$ is strictly $(p-1)$-convex.
In the sequel, we always choose such an attracting neighborhood $U.$

Up to an automorphism of $\mathbb P^k,$ $I$ is defined in
homogeneous coordinates by $I=\{x\in\mathbb P^k\,|\,
x_i=0,\ 0\leq i\leq k-p\}.$ The function
$$\eta(x)=\frac{|x_{k-p+1}|^2+\cdots+|x_{k}|^2}{|x_0|^2+\cdots+|x_{k-p}|^2},$$
is a $(q+1)$-convex exhausting function of $\mathbb P^k\setminus I,$ i.e.
$\mathbb P^k\setminus I$ is \textit{completely $q$-convex}. In general,
strictly $q$-convex subsets of a completely $q$-convex domain are not
completely strictly $q$-convex. However, in our case it is easy to construct
from a $q$-convex function $\rho$ such that
$$U\cap V=\{x\in V\,|\, \rho(x)<0\}$$
for some neighborhood $V$ of $\partial U,$ a $q$-convex defining function
defined in a neighborhood of $\overline U.$ Indeed, it is enough to compose
$(\eta,\rho)$
with a good approximation of the maximum function (see
\cite[Definition 4.12]{henkin-leiterer}). It will give a $(q+1)$-convex
function since the positive eigenvalues of the complex Hessians of $\rho$ and
$\eta$ are in the same directions. Therefore, $U$ is completely strictly
$(p-1)$-convex and we have the following solution for the $\ddc$-equation in
symmetric bidegrees.
\begin{theorem}\label{th ddc}
Let $U$ be as above. If $\varphi$ is a $\mathcal C^2$ $(r,r)$-form in a
neighborhood of
$\overline U$ with $k-p\leq r\leq k,$ then there exists a continuous
$(r,r)$-form $\psi$ on $U$ such that $\ddc\psi=\ddc\varphi$ and
$$\|\psi\|_{\infty,U}\leq C\|\ddc\varphi\|_{\infty,U}$$
for some $C>0$ independent of $\varphi.$
\end{theorem}
\begin{proof}
The proof follows closely the proof of Theorem 2.7 in \cite{dns-hori}.

Without loss of generality, we can assume that $\varphi$ is real and therefore
$\ddc \varphi$ is also real. First, we solve the equation
$d\xi=\ddc\varphi$ with estimates. Let $W$ be a small neighborhood of
$\overline U,$ with the same geometric property and such that $\varphi$ is
defined on $W.$
The maps $A_\theta$ defined in Section
\ref{sec disc} give a homotopy $A: [0,1]\times W\to W,$
$A(\theta,x)=A_\theta(x),$ between $A_1=\id$ and the projection $A_0$ of $W$ on
$L.$ Since $L$ has dimension $k-p,$ $A_0^*$ vanish identically on
$(r+1,r+1)$-forms if $r\geq k-p.$ Therefore, by homotopy formula (see e.g
\cite[p38]{bott-tu}), there exists a form $\xi$ on $W$ such that $d\xi=\ddc
\varphi$
and $\|\xi\|_{\infty,U}\lesssim \|\ddc\varphi\|_{\infty,U}.$
Moreover, possibly by exchanging $\xi$ by $(\xi+\overline\xi)/2,$ we can assume
that $\xi=\Xi+\overline\Xi$ where $\Xi$ is a $(r,r+1)$-form. As $d\xi$ is a
$(r+1,r+1)$-form, it follows that $\dbar
\Xi=0$ and $d\xi=\partial\Xi+\dbar\overline\Xi.$ Therefore, by
Theorem \ref{th vanish}, $\Xi$ is $\dbar$-exact and by Theorem \ref{th hl},
there
exists a continuous $(r,r)$-form $\Psi$ such that $\dbar\Psi=\Xi$ and
$\|\Psi\|_{\infty,U}\lesssim\|\Xi\|_{\infty,U}.$

Finally, if $\psi=-i\pi(\Psi-\overline\Psi)$ we have
$$\ddc
\psi=\ddbar(\Psi-\overline\Psi)=\partial\Xi+\dbar\overline\Xi=\ddc\varphi,$$
and
$$\|\psi\|_{\infty,U}\lesssim\|\Xi\|_{\infty,U}\lesssim\|\ddc\varphi\|_{\infty,U
}.$$
\end{proof}

\section{Attracting speed}

For $R$ in $\Cc_p(U),$ we denote by $R_n$ its normalized push-forward by $f^n,$
i.e. $R_n:=d^{-(k-p)n}(f^n)_*(R).$ To obtain \eqref{eq main}, the first
observation
is
that the norm of $R_n-\tau,$ seen as a linear form on the space of continuous
test $(k-p,k-p)$-forms, is
bounded independently of $n$ and $R.$ Therefore, it is sufficient to establish
\eqref{eq main} for
$\alpha=2$ and then apply interpolation theory between Banach spaces, see e.g.
\cite{triebel-interpolation}, in order to obtain the general case.

Let denote by $X$ the set of all real continuous $(k-p,k-p)$-forms $\phi$ on
$U$ such that $\ddc\phi=0$ and $|\langle R-\tau,\phi\rangle|\leq1$ for all
$R\in\Cc_p(U).$ Observe that, since $f(U)\Subset U,$ if $\phi$
is in $X$ then $f^*(\phi)$ is defined on $U$ where it is still a real
continuous form with $\ddc(f^*(\phi))=0.$ The set $X$ is a truncated convex
cone and the first part of the proof of Theorem \ref{th main} is to show that
$d^{-(k-p)}f^*$ acts by contraction on it. This result is available without any
assumption on $\|\wedge^{k-p+1}Df\|.$ It is based on Lemma \ref{le harmo} and
Harnack's inequality for harmonic functions.

\begin{lemma}\label{le harnack}
 There exists a constant $0<\lambda_1<1$ such that for any $R$ in $\Cc_p(U),$
$\phi$ in $X$
and $n$ in $\mathbb N$ we have
$$|\langle R_n-\tau,\phi\rangle|\leq \lambda_1^n.$$
\end{lemma}
\begin{proof}
If $R$ is in $\Cc_p(U)$ and $\phi$ in $X,$ $R_1:=d^{-(k-p)}f_*(R)$ is in
$\Cc_p(U')$ and we define the function $h_{R,\phi}$ on
$V$ by $h_{R,\phi}(\theta):=\langle R_{1,\theta}-\tau,\phi\rangle,$ where
$\theta\mapsto  R_{1,\theta}$ is the structural disc described in Section
\ref{sec disc}.
The definition of $X$ implies that $|h_{R,\phi}|\leq1$ on $V,$ for all $R\in
\Cc_p(U)$ and $\phi\in X.$ Moreover, since $R_1$ is in $\Cc_p(U'),$ it
follows
from
Lemma \ref{le harmo} that all these functions are harmonic on $V.$ 

Now, observe that if we take $R=\tau$ then $h_{\tau,\phi}(1)=0$ for all $\phi\in
X,$ since $d^{p-k}f_*\tau=\tau.$ Hence, as $|h_{\tau,\phi}|\leq1$ on $V,$
Harnack's
inequality says that there exists $0\leq a<1$ such that
$|h_{\tau,\phi}(0)|\leq a$ for all $\phi$ in $X.$ On the other hand,
$R_{1,0}$ is a current independent of $R.$ So, for all $R\in \Cc_p(U)$ and
$\phi\in X$ we have $h_{R,\phi}(0)=h_{\tau,\phi}(0)$ and therefore
$|h_{R,\phi}(0)|\leq a.$ Once again, we deduce from Harnack's inequality there
exists $0<\lambda_1<1,$ independent of $R$ and $\phi,$ such that
$|h_{R,\phi}(1)|\leq\lambda_1$ or equivalently 
$$\left|\left\langle
R_1-\tau,\frac{\phi}{\lambda_1}\right\rangle\right|=|\langle R-\tau,
\phi_1\rangle|\leq1,$$
where $\phi_1=d^{-(k-p)}f^*(\phi/\lambda_1).$ Moreover, $\phi_1$ is defined on
$U$ and $\ddc\phi_1=0.$ It follows that $\phi_1$ is in $X.$ Using the same
arguments with $\phi_1$ instead of $\phi$ gives that $|\langle
R_1-\tau,\phi_1\rangle|\leq\lambda_1$ which can be rewrite
$|\langle R_2-\tau,\phi\rangle|\leq \lambda_1^2.$ Inductively, we obtain that
$|\langle
R_n-\tau,\phi\rangle|\leq \lambda_1^n.$
\end{proof}
\begin{remark}
The constant $\lambda_1$ is not directly related to $f.$ Indeed, it only
depends on $V$ i.e. on the size of $U$ and the distance between $\partial U$ and
$\partial f(U).$
If $h$ is the unique biholomorphism between $V$ and the unit disc in $\mathbb C$
such
that $h(0)=0$ and $h(1)=\alpha\in ]0,1[$ then Harnack's inequality gives
explicitly that we can take, in the proof above, 
$a=2\alpha/(1+\alpha)$ and $\lambda_1=4\alpha/(1+\alpha)^2.$
\end{remark}
In order to prove Theorem \ref{th main}, we use Theorem \ref{th ddc} together
with
the assumption on $\|\wedge^{k-p+1}Df\|$ and Lemma \ref{le harnack}.

If $\|\wedge^{k-p+1}Df(z)\|<1$ on
$\overline U$ then by continuity, there exists a constant
$0<\lambda_2<1$ such that $\|\wedge^{k-p+1}Df(z)\|<\lambda_2$ on $U.$
Hence, if $\varphi$ is a $(k-p,k-p)$-form of class $\mathcal C^2,$ we have
for $\varphi_i:=d^{-i(k-p)}(f^i)^*(\varphi)$ with $i\in\mathbb N$
$$\|\ddc\varphi_i\|_{\infty,U}\lesssim
\frac{\lambda_2^{2i}}{d^{i(k-p)}}\|\varphi\|_{\mathcal C^2}.$$
Here, the symbol $\lesssim$ means inequality up to a constant which only
depends on our conventions and on $U.$ By Theorem \ref{th ddc} with $r=k-p$,
there
exists a continuous $(k-p,k-p)$-form $\psi_i$ on $U$ such that
$$\ddc\psi_i=\ddc \varphi_i$$
and
$$\|\psi_i\|_{\infty,U}\lesssim \|\ddc\varphi_i\|_{\infty,U}\lesssim
\frac{\lambda_2^{2i}}{d^{i(k-p)}}\|\varphi\|_{\mathcal C^2}.$$

We can now complete the proof of our main result.
\begin{proof}[End of the proof of Theorem \ref{th main}]
Let $R$ be in $\Cc_p(U)$ and $\varphi$ be a $(k-p,k-p)$-form of class $\mathcal
C^2.$ Without loss of generality, we can assume that $\varphi$ is real. Let
$0\leq i\leq n$ be two arbitrary integers. We set $l:=n-i.$ 
If $R_n,$ $\varphi_i$ and $\psi_i$ are defined as above then we have
\begin{equation*}
\langle R_n-\tau,\varphi\rangle=\langle
R_l-\tau,\varphi_i\rangle=\langle R_l-\tau,\varphi_i-\psi_i\rangle+\langle
R_l-\tau,\psi_i\rangle,
\end{equation*}
since $\tau$ is invariant. On the one hand, 
\begin{equation}\label{eq decoup1}
\langle R_l-\tau,\psi_i\rangle\lesssim
\|\psi_i\|_{\infty,U}\lesssim
\frac{\lambda_2^{2i}}{d^{i(k-p)}}\|\varphi\|_{\mathcal C^2},
\end{equation}
since $R_l$ and $\tau$ are supported on $U.$ On the other hand, observe that
there exists a constant $M\geq1$ independent of $\varphi$ such that
$\|d^{-(k-p)}f^*(\varphi)\|_\infty\leq M\|\varphi\|_\infty.$ Therefore,
\begin{align*}
\|\varphi_i-\psi_i\|_{\infty,U}\leq
M^i\|\varphi\|_{\infty}+\|\psi_i\|_{\infty, U}&\leq
M^i\|\varphi\|_{\infty}+C\frac{\lambda_2^{2i}}{d^{i(k-p)}}\|\varphi\|_{\mathcal
C^2}\\
&\lesssim M^i\|\varphi\|_{\mathcal C^2},
\end{align*}
and in particular
$$|\langle S-\tau,\varphi_i-\psi_i\rangle|\lesssim M^i\|\varphi\|_{\mathcal
C^2},$$
for any $S$ in $\Cc_p(U).$

Moreover, $\varphi_i-\psi_i$ is a real continuous
$(k-p,k-p)$-form on $U$ and $\ddc(\varphi_i-\psi_i)=0.$ Hence,
$(\varphi_i-\psi_i)/(CM^i\|\varphi\|_{\mathcal
C^2})$ belongs to $X$ where $C>0$ is a constant depending only on $U$ and on
our conventions. It follows by Lemma \ref{le harnack} that
\begin{equation}\label{eq decoup2}
 |\langle R_l-\tau,\varphi_i-\psi_i\rangle|\leq CM^i\|\varphi\|_{\mathcal
C^2}\lambda_1^l.
\end{equation}
To summarize, equations \eqref{eq decoup1} and \eqref{eq decoup2} imply that
there are constants $0<\lambda_1,\lambda_2<1,$ and $M\geq 1$
such that
$$|\langle R_n-\tau,\varphi\rangle|\lesssim \|\varphi\|_{\mathcal
C^2}\left(M^i\lambda_1^l+\frac{\lambda_2^{2i}}{d^{i(k-p)}}\right).$$
If $q\in\mathbb N$ is large enough then $M\lambda_1^q<1.$ Therefore, if we
choose $n\simeq(q+1)i$, we obtain $l\simeq iq$ and
$$|\langle R_n-\tau,\varphi\rangle|\lesssim \|\varphi\|_{\mathcal
C^2}\lambda^n,$$
where $\lambda:=\max(\lambda_2^2d^{-(k-p)},M\lambda_1^q)^{1/(q+1)}<1.$ This
estimate holds for
arbitrary $n$ in $\mathbb N$ and is uniform on $\varphi$ and $R.$
\end{proof}
\begin{remark}
 In Theorem \ref{th main}, it is enough to assume that
$\|\wedge^{k-p+1}Df(z)\|< d^{(k-p)/2}$ on $\overline U.$ Moreover, it is easy
using small perturbations of a suitable polynomial map to construct examples
with $\|\wedge^{k-p+1}Df(z)\|$ as small as we want on $\overline U.$
\end{remark}

\section{Hyperbolicity of the equilibrium measure}
In this section, we prove Theorem \ref{th hyp}. Recall that the equilibrium
measure associated to $A$ is given by $\nu:=\tau\wedge T^{k-p}.$ It has maximal
entropy on $A$ equal to $(k-p)\log d,$ \cite{d-attractor}. On the other hand, we
have the following powerful result, see \cite{dethelin-lyap} and
\cite{dupont-large-entropy}.
\begin{theorem}\label{th dT}
If the Lyapunov exponents of $\nu$ are ordered so that
$$\chi_1\geq\cdots\geq
\chi_{a-1}>\chi_a\geq\cdots\geq\chi_k,$$
then
\begin{equation}\label{eq entropy}
h(\nu)\leq (a-1)\log d+2\sum_{i=a}^k\chi_i^+,
\end{equation}
where $h(\nu)$ denotes the entropy of $\nu$ and $\chi_i^+:=\max(\chi_i,0).$
\end{theorem}
Now, let $1\leq c\leq k$ be such that
$$\chi_1\geq\cdots\geq
\chi_c>0\geq\chi_{c+1}\geq\cdots\geq\chi_k.$$
If we take $a=c+1$ in Theorem \ref{th dT}, we obtain $h(\nu)\leq c\log d.$
Since $h(\nu)=(k-p)\log d,$ it follows that $c\geq (k-p).$ It means there
are at least $k-p$ strictly positive Lyapunov exponents. Moreover, if we have
equality, $c=k-p,$ Theorem \ref{th dT} applied to the smallest $a$ such that
$\chi_a=\chi_c$ gives
$$(k-p)\log d=h(\nu)\leq (a-1)\log d+2(k-p-a+1)\chi_c.$$
Hence, $\chi_c\geq (\log d)/2.$ Note that in this part we do not need the
assumption on $\|\wedge^{k-p+1}Df\|.$

It remains to prove that the assumptions of Theorem \ref{th main} imply that
$c\leq k-p$ and $\chi_{c+1}<-(k-p)(\log d)/2.$ It is not hard to deduce form
Oseledec theorem \cite{oseledec} that the sum of the $q$ largest Lyapunov
exponents
verifies
$$\chi_1+\cdots+\chi_q=\lim_{n\to\infty}\frac{1}{n}\log\|\wedge^qDf^n(z)\|,
$$
for $\nu$-almost all $z.$ Moreover, we have
$$\|\wedge^qDf^{n+m}(z)\|\leq \|\wedge^qDf^n(z)\|\|\wedge^qDf^m(f^n(z))\|.$$
Therefore, it follows that
$$\|\wedge^qDf^{n}(z)\|\leq (\max_{z\in U}\|\wedge^qDf(z)\|)^n$$
and
$$\chi_1+\cdots+\chi_q\leq\log\max_{z\in U}\|\wedge^qDf(z)\|=:\gamma.$$
Hence, if $\|\wedge^{k-p+1}Df(z)\|<1$ on $\overline U$ then
$$\chi_1+\cdots+\chi_{k-p+1}\leq\gamma<0.$$
Therefore, $c\leq k-p$ and we have seen above that in this case $c=k-p$ and
$\chi_c\geq(\log d)/2.$ Finally, we have
$$\gamma\geq\chi_1+\cdots+\chi_{k-p}+\chi_{k-p+1}\geq \frac{k-p}{2}\log
d+\chi_{k-p+1},$$
which implies 
$$\chi_{k-p+1}\leq\gamma-\frac{k-p}{2}\log d.$$
\begin{remark}
 Theorem \ref{th dT} with $a=1$ implies the Ruelle inequality, i.e.
$$\chi_1+\cdots+\chi_c\geq\frac{k-p}{2}\log d.$$
Therefore, it is enough to assume that
$\|\wedge^{k-p+1}Df(z)\|<d^{(\frac{k-p}{2})(\frac{k-p+1}{k})}$
on $\overline U$ since
$$\chi_1+\cdots+\chi_{k-p+1}\geq \frac{k-p+1}{c}(\chi_1+\cdots+\chi_c),$$
if $c\geq k-p+1.$
\end{remark}

\bibliographystyle{alpha}

\noindent
J. Taflin, UPMC Univ Paris 06, UMR 7586, Institut de
Math{\'e}matiques de Jussieu, F-75005 Paris, France. {\tt 
taflin@math.jussieu.fr} \\
Universitetet i Oslo, Mathematisk Institutt, Postok 1053 Blindern, 0316 Oslo,
Norway. {\tt johantaf@math.uio.no}
\end{document}